\documentclass[11pt]{amsart}
\usepackage{amssymb}
\usepackage[english]{babel}
\usepackage{amsmath,amssymb}
\usepackage{amsthm}
\usepackage{amsmath}
\usepackage{graphicx}
\usepackage{amsfonts}

\providecommand{\U}[1]{\protect\rule{.1in}{.1in}}
\providecommand{\U}[1]{\protect\rule{.1in}{.1in}}
\newtheorem{Th}{Theorem}[section]
\newtheorem{Lem}[Th]{Lemma}
\newtheorem{Cor}[Th]{Corollary}
\newtheorem{Prop}[Th]{Proposition}
\newtheorem{Def-Prop}[Th]{Definition-Proposition}
\newtheorem{conj}[Th]{Conjecture}
\theoremstyle{definition}

\newtheorem{Exa}[Th]{Example}

\theoremstyle{remark}
\newtheorem{Rem}[Th]{Remark}

{\rmfamily}

\begin{document}
\title[Isomorphic induced modules and Dynkin diagram automorphisms]{%
Isomorphic induced modules and Dynkin diagram automorphisms of semisimple
Lie algebras}
\date{}
\author{J\'{e}r\'{e}mie Guilhot and C\'{e}dric Lecouvey}
\maketitle

\begin{abstract}
Consider a simple Lie algebra $\mathfrak{g}$ and $\overline{\mathfrak{g}}%
\subset \mathfrak{g}$ a Levi subalgebra. Two irreducible $\overline{%
\mathfrak{g}}$-modules yield isomorphic inductions to $\mathfrak{g}$ when
their highest weights coincide up to conjugation by an element of the Weyl
group $W$ of $\mathfrak{g}$ which is also a Dynkin diagram automorphism of $%
\overline{\mathfrak{g}}$. In this paper we study the converse problem: given
two irreducible $\overline{\mathfrak{g}}$-modules of highest weight $\mu $
and $\nu $ whose inductions to $\mathfrak{g}$ are isomorphic, can we
conclude that $\mu $ and $\nu $ are conjugate under the action of an element
of $W$ which is also a Dynkin diagram automorphism of $\overline{\mathfrak{g}%
}$ ? We conjecture this is true in general. We prove this conjecture in type 
$A$ and, for the other root systems, in various situations providing $\mu $
and $\nu $ satisfy additional hypotheses. Our result can be interpreted as
an analogue for branching coefficient of the main result of \cite{Raj} on
tensor product multiplicities. 
\end{abstract}

\section{Introduction}

Let $\mathfrak{g}$ be a simple Lie algebra over $\mathbb{C}$ and $\overline{%
\mathfrak{g}}$ a Levi subalgebra. Let $\mu$ and $\nu$ be two dominant
integral weights for $\overline{\mathfrak{g}}$.\ Denote by $\overline{V}%
(\mu) $ and $\overline{V}(\nu)$ the associated highest weight $\overline{%
\mathfrak{g}}$-modules. Let $\overline{V}(\mu)\uparrow _{\overline{\mathfrak{%
g}}}^{\mathfrak{g}}$ and $\overline{V}(\nu )\uparrow_{\overline{\mathfrak{g}}%
}^{\mathfrak{g}}$ be the $\mathfrak{g}$-modules obtained by induction from $%
\overline{\mathfrak{g}}$.\ When $\mu$ and $\nu$ are conjugate by an element
of the Weyl group $W$ of $\mathfrak{g}$ which is also a Dynkin diagram
automorphism of $\overline{\mathfrak{g}}$, the modules $\overline{V}%
(\mu)\uparrow_{\overline{\mathfrak{g}}}^{\mathfrak{g}}$ and $\overline{V}%
(\nu)\uparrow_{\overline{\mathfrak{g}}}^{\mathfrak{g}}$ are isomorphic; see
Proposition \ref{Prop_direct}. In this paper, we address the following
question: assume $\overline{V}(\mu)\uparrow_{\overline{\mathfrak{g}}}^{%
\mathfrak{g}}$ and $\overline{V}(\nu)\uparrow_{\overline{\mathfrak{g}}}^{%
\mathfrak{g}}$ are isomorphic, can we conclude that $\mu$ and $\nu$ are
conjugate by an element of the Weyl group $W$ of $\mathfrak{g}$ which is
also a Dynkin diagram automorphism of $\overline{\mathfrak{g}}$ ?

It is interesting to reformulate this problem in terms of the (infinite)
matrix $M=(m_{\mu}^{\lambda})$ with columns and rows labelled respectively
by the dominant weights $\lambda$ of $\mathfrak{g}$ and by the dominant
weights $\mu$ of $\overline{\mathfrak{g}}$. Here $m_{\mu}^{\lambda}$ denotes
the branching coefficient corresponding to the multiplicity of the
irreducible highest weight $\mathfrak{g}$-module $V(\lambda)$ in $\overline{V%
}(\mu)\uparrow_{\overline{\mathfrak{g}}}^{\mathfrak{g}}$ (or equivalently
the multiplicity of $\overline{V}(\mu)$ in the restriction of $V(\lambda)$
to $\overline{\mathfrak{g}}$).\ We then ask if two rows of the matrix $M$
can be equal. Note that two distinct columns of $M$ labelled by $\lambda$
and $\mu$ of $M$ cannot coincide since this would imply $V(\lambda)\simeq
V(\mu)$. Indeed both modules would then have the same weight decomposition
and therefore the same characters.

The matrix $M$ contains inner multiplicities associated to simple $\mathfrak{%
g}$-modules. We can also address a similar question for the outer (tensor
product) multiplicities.\ The corresponding matrix, say $C$, has columns and
rows labelled by dominant weights of $\mathfrak{g}$ and $k$-tuples $%
(\mu^{(1)},\ldots,\mu^{(k)})$ of such dominant weights.\ The coefficients $%
c_{\mu^{(1)},\ldots,\mu^{(k)}}^{\lambda}$ is then the multiplicity of $%
V(\lambda)$ in $V(\mu^{(1)})\otimes\cdots\otimes V(\mu^{(k)})$. It was
proved by Rajan in \cite{Raj} (see also \cite{VV} for a shorter proof and an
extension to the case of Kac-Moody algebras) that two rows of $C$ are equal
if and only if the associated $k$-tuples of dominant weights coincide up to
permutation. It is also easy to see that if the columns of $C$ labelled by $%
\lambda$ and $\kappa$ coincide, then $\lambda=\kappa$ (take $(\mu
^{(1)},\ldots,\mu^{(k)})=(\lambda,0,\ldots,0)$ and $(\mu^{(1)},\ldots
,\mu^{(k)})=(\kappa,0,\ldots,0)$).

Finally, one can also consider the decomposition matrix $D$ associated to
the modular representation theory of the symmetric group in characteristic $%
p $. Its columns and rows are indexed by $p$-restricted partitions and
partitions of $n$, respectively. The study of possible identical rows and
columns was considered by Wildon in \cite{Wil}: the columns of $D$ are
distinct and its rows can only coincide in characteristic $2$ when the
underlying partitions are conjugate.

\bigskip

In the present paper, we prove that two rows of the matrix $M$ corresponding
to weights conjugate by an element of the Weyl group $W$ of $\mathfrak{g}$
which is also a Dynkin diagram automorphism of $\overline{\mathfrak{g}}$
coincide. We conjecture that the converse is true and prove this conjecture
in various cases (see Theorem~\ref{Th_final}). We believe that the study of
the matrix $M$ is more complicated than that of the matrix $C$ for two main
reasons. First, there could exist infinitely many nonzero coefficients in a
row of $M$ (this is not the case for $C$). Second, the possible
transformations relating the labels corresponding to identical rows in $M$
are more complicated than in the case of the matrix $C$ (where they simply
correspond to permutations of the $k$-tuples of dominant weights).

\bigskip

The paper is organised as follows.\ Section 2 is devoted to some classical
background on representation theory of Lie algebras. In Section 3, we study
the relationships between the roots and the weights of $\mathfrak{g}$ and $%
\overline{\mathfrak{g}}$. In Section 4, we formulate our conjecture in terms
of equality of distinguish functions in the character ring of $\mathfrak{g}$%
. This allows us in Section 5 to prove our conjecture when $\mu$ and $\nu$
are far enough from the walls of the Weyl chamber in which they appear. In
Section~6, we prove the conjecture when $\mu+2\overline{\rho}$ or $\nu+2%
\overline{\rho}$ (where $\overline{\rho}$ denotes the half sum of positive
roots of $\overline{\mathfrak{g}}$) is dominant for $\mathfrak{g}$. Finally,
in Section 7, we prove the conjecture in the case $\mathfrak{g=gl}_{n}$ by
using the main result of Rajan \cite{Raj}. This also permits to establish it
when $\mathfrak{g}$ is a classical Lie algebra of type $B_{n},C_{n}$ or $%
D_{n}$ and $\overline{\mathfrak{g}}\mathfrak{=gl}_{n}$.


\section{Background on Lie algebras}

Let $\mathfrak{g}$ be a simple Lie algebra over $\mathbb{C}$ with triangular
decomposition 
\begin{equation*}
\mathfrak{g=}\bigoplus\limits_{\alpha \in R_{+}}\mathfrak{g}_{\alpha }\oplus 
\mathfrak{h}\oplus \bigoplus\limits_{\alpha \in R_{+}}\mathfrak{g}_{-\alpha }
\end{equation*}%
so that $\mathfrak{h}$ is the Cartan subalgebra of $\mathfrak{g}$ and $R_{+}$
its set of positive roots. The root system $R=R_{+}\sqcup (-R_{+})$ of $%
\mathfrak{g}$ is realised in a real Euclidean space $E$ with inner product $%
(\cdot ,\cdot )$.\ For any $\alpha \in R,$ we write $\alpha ^{\vee }=\frac{%
2\alpha }{(\alpha ,\alpha )}$ for its coroot. Let $S\subset R_{+}$ be the
subset of simple roots. The set $P$ of integral weights for $\mathfrak{g}$
satisfies $(\beta ,\alpha ^{\vee })\in \mathbb{Z}$ for any $\beta \in P$ and 
$\alpha \in R$. We write $P_{+}=\{\beta \in P\mid (\beta ,\alpha ^{\vee
})\geq 0$ for any $\alpha \in S\}$ for the cone of dominant weights of $%
\mathfrak{g}$. Let $W$ be the Weyl group of $\mathfrak{g}$ generated by the
reflections $s_{\alpha }$ with $\alpha \in R_{+}$ (or equivalently by the
simple reflections $s_{\alpha }$ with $\alpha \in S$). Set $C=\{x\in E\mid
(x,\alpha ^{\vee })>0\}$ and $\overline{C}=\{x\in E\mid (x,\alpha ^{\vee
})\geq 0\}$.\ For any $w\in W$, we set 
\begin{equation*}
C_{w}=w^{-1}(C),\quad \overline{C}_{w}=w^{-1}(\overline{C})\quad \text{and}%
\quad P_{+}^{w}=P\cap \overline{C}_{w}.
\end{equation*}%
Each set $w^{-1}(S)$ can be chosen as a set of simple roots for $R$, the
corresponding set of positive roots is then $R_{+}^{w}=w^{-1}(R_{+})$. Given 
$w\in W$, we define the dominance order $\leq _{w}$ on $P$ by the following
relation: $\gamma \leq _{w}\beta $ if and only if $\beta -\gamma $
decomposes as a sum of roots in $R_{+}^{w}$. When $w=1$, we simply write $%
\leq $ for the order $\leq _{1}$.

Now consider a subset of simple roots $\overline{S}\subset S$.\ Write $%
\overline{R}\subset R$ for the parabolic root system generated by $\overline{%
S}$ and $\overline{R}_{+}=\overline{R}\cap R_{+}$ the corresponding set of
positive roots.$\ $Let $\overline{\mathfrak{g}}\subset\mathfrak{g}$ be the
Levi subalgebra of $\mathfrak{g}$ with set of positive roots $\overline {R}%
_{+}$ and triangular decomposition 
\begin{equation*}
\overline{\mathfrak{g}}\mathfrak{=}\bigoplus\limits_{\alpha\in\overline{R}%
_{+}}\mathfrak{g}_{\alpha}\oplus\mathfrak{h}\oplus\bigoplus\limits_{\alpha
\in\overline{R}_{+}}\mathfrak{g}_{-\alpha}. 
\end{equation*}
In particular, $\mathfrak{g}$ and $\overline{\mathfrak{g}}$ have the same
Cartan subalgebra. The algebras $\mathfrak{g}$ and $\overline{\mathfrak{g}}$
have the same integral weight lattice $P$.\ The Weyl group $\overline{W}$ of 
$\overline{\mathfrak{g}}$ is generated by the simple reflections $s_{\alpha} 
$ with $\alpha\in\overline{S}$. Denote by $\overline{P}_{+}\subset P$ the
set of dominant integral weights of $\overline{\mathfrak{g}}$. We shall also
need the partial order $\preceq$ on $P$ defined by the following relation: $%
\gamma\preceq\beta$ if and only if $\beta-\gamma$ decomposes as a sum of
roots in $\overline{R}_{+}$.

\begin{Exa}
Consider $\mathfrak{g=sp}_{12}$. We have 
\begin{equation*}
R_{+}=\{\varepsilon_{i}-\varepsilon_{j}\mid1\leq i<j\leq6\}\cup\{\varepsilon
_{i}+\varepsilon_{j} \mid1< i< j\leq6\}\cup\{2\varepsilon_{i}\mid1\leq
i\leq6 \} 
\end{equation*}
and 
\begin{equation*}
P_{+}=\{x=(x_{1},\ldots,x_{6})\in\mathbb{Z} ^{6}\mid x_{1}\geq\cdots\geq
x_{6}\geq0\}. 
\end{equation*}
The Levi subalgebra $\overline{\mathfrak{g}}\subset\mathfrak{g}$ such that 
\begin{equation*}
\overline{R}_{+}=\{\varepsilon_{1}-\varepsilon_{2},\varepsilon_{1}-%
\varepsilon_{3},\varepsilon_{2}-\varepsilon_{3}\}\cup\{\varepsilon_{4}\pm%
\varepsilon_{5},\varepsilon_{4}\pm\varepsilon_{6},\varepsilon_{5}\pm%
\varepsilon_{6}\}\cup\{2\varepsilon_{4},2\varepsilon_{5},2\varepsilon _{6}\} 
\end{equation*}
is then isomorphic to $\mathfrak{gl}_{3}\oplus\mathfrak{sp}_{6}$.
\end{Exa}

Given $\lambda\in P_{+}$, we denote by $V(\lambda)$ the finite dimensional
irreducible representation of $\mathfrak{g}$ with highest weight $\lambda $.$%
\ $Let $s_{\lambda}$ be the character of $V(\lambda)$. This is an element of
the group algebra $\mathbb{Z}[P]$ with basis $\{e^{\beta}\mid\beta\in P\}$.\
More precisely 
\begin{equation*}
s_{\lambda}=\sum_{\mu\in P}\dim V(\lambda)_{\mu}e^{\mu}
\end{equation*}
where $V(\lambda)_{\mu}$ is the weight space in $V(\lambda)$ corresponding
to $\mu$. Set $\mathbb{G=Z}^{W}[P]$. We then have $s_{\lambda}\in\mathbb{G}$%
, that is $s_{\lambda}$ is symmetric under the action of $W$. We also recall
the Weyl character formula%
\begin{equation*}
s_{\lambda}=\dfrac{\sum_{w\in W}\varepsilon(w)e^{w(\lambda+\rho)-\rho}}{%
\prod_{\alpha\in R_{+}}(1-e^{-\alpha})}
\end{equation*}
where $\rho=\frac{1}{2}\sum_{\alpha\in R_{+}}\alpha$. Note that, for any $%
w\in W$ and $\beta\in P$, we have $s_{w(\beta)}=\varepsilon(w)s_{w\circ\beta}
$ where $\circ$ is the dot action of the Weyl group defined by $w\circ
\beta=w(\beta+\rho)-\rho$.

Using the restriction of $V(\lambda)$ to $\overline{\mathfrak{g}}$ we define
the branching coefficients $m_{\mu}^{\lambda}$ by 
\begin{equation*}
s_{\lambda}=\sum_{\mu\in\overline{P}_{+}}m_{\mu}^{\lambda}\overline{s}_{\mu}
\end{equation*}
where $\overline{s}_{\mu}$ is the character of the irreducible
representation $\overline{V}(\mu)$ of $\overline{\mathfrak{g}}$ of highest
weight $\mu$.\ We introduce the partition function $\overline{\mathcal{P}}$
defined by 
\begin{equation*}
\prod_{\alpha\in R_{+}\setminus\overline{R}_{+}}\frac{1}{1-e^{\alpha}}%
=\sum_{\beta\in P}\overline{\mathcal{P}}(\beta)e^{\beta}. 
\end{equation*}
Then, the branching coefficient $m_{\mu}^{\lambda}$ can be computed in term
of $\overline{\mathcal{P}}$ using the Weyl character formula (see \cite[p.
357]{GW}).

\begin{Th}
\label{Th_multi}Let $\lambda\in P_{+}$ and $\mu\in\overline{P}_{+}$.\ Then 
\begin{equation*}
m_{\mu}^{\lambda}=\sum_{w\in W}\varepsilon(w)\overline{\mathcal{P}}%
(w(\lambda+\rho)-\mu-\rho) 
\end{equation*}
where $\varepsilon$ is the sign representation of $W$.
\end{Th}


\section{Dominant weights of $\overline{\mathfrak{g}}$ and Weyl chambers}

\label{section3}

This section is devoted to study the relationship between the various
subsets of roots and weights we have defined. To this end we introduce the
following subset which will play an important role in this paper: 
\begin{equation*}
U=\{u\in W\mid u(\overline{R}_{+})\subset R_{+}\}.   \label{defU}
\end{equation*}

\begin{Prop}
\label{Lem-U1} We have

\begin{enumerate}
\item 
\begin{equation*}
\overline{P}_{+}=\bigcup\limits_{u\in U}u^{-1}(P_{+}). 
\end{equation*}

\item 
\begin{equation*}
\overline{R}_{+}=\bigcap\limits_{u\in U}u^{-1}(R_{+}). 
\end{equation*}

\item Each element $w$ in $W$ admits a unique decomposition under the form $%
w=u\overline{w}$ with $u\in U$ and $\overline{w}\in\overline{W}$.
\end{enumerate}
\end{Prop}

\begin{proof}
We prove 1. Let $\lambda \in P_{+}$ and $u\in U$. For all $\alpha \in 
\overline{R}_{+}$, we have 
\begin{equation*}
(u^{-1}(\lambda ),\alpha ^{\vee })=(\lambda ,u(\alpha )^{\vee })\geq 0
\end{equation*}%
since $\lambda \in P_{+}$ and $u(\alpha )\in R_{+}$. It follows that $%
u^{-1}(\lambda )\in \overline{P}_{+}$ and $\bigcup\limits_{u\in
U}u^{-1}(P_{+})\subset \overline{P}_{+}$. \newline

Next let $\gamma \in \overline{P}_{+}$. There exists $u^{\prime }\in W$ such
that $u^{\prime }(\gamma )\in P_{+}$. Let $\alpha \in \overline{R}_{+}$.
Then $(\gamma ,\alpha )=(u^{\prime }(\gamma ),u^{\prime }(\alpha ))\geq 0$.
If the inequality is strict then we have $u^{\prime }(\alpha )\in R_{+}$. We
set 
\begin{align*}
R_{>0}& :=\{\beta \in R\mid (u^{\prime }(\gamma ),\beta )>0\}\subset R_{+},
\\
R_{0}& :=\{\beta \in R\mid (u^{\prime }(\gamma ),\beta )=0\} \\
R_{0,+}& :=\{\beta \in R_{+}\mid (u^{\prime }(\gamma ),\beta )=0\}\text{, }%
R_{0,-}=-R_{0,+}.
\end{align*}%
Note that $R_{0}$ is a subroot system of $R$ and that the simple system
associated to $R_{0,+}$ consists simply of $R_{0,+}\cap S$. Also, since $%
u(\gamma )\in P_{+}$, we have $R_{+}=R_{>0}\cup R_{0,+}$. Let $W_{0}=\langle
s_{\beta }\mid \beta \in R_{0}\rangle $. The group $W_{0}$ then acts on $R$
and stabilises both $R_{0}$ and $R_{>0}$. Since all the roots in $R_{0}$ are
orthogonal to $u^{\prime }(\gamma )$ we have $vu^{\prime }(\gamma
)=u^{\prime }(\gamma )\in P_{+}$ for all $v\in W_{0}$. Now let $u$ be the
element of minimal length in the coset $W_{0}u^{\prime }$. By the previous
argument, we do have $u^{\prime }(\gamma )\in P_{+}$. Let us show that $u\in
U$. Let $\alpha \in \overline{R}_{+}$. First if $u^{\prime }(\alpha )\in
R_{>0}$, then so does $u(\alpha )$ since $W_{0}$ stabilises $R_{>0}$ and we
are done in this case since $u(\alpha )\in R_{>0}\subset R_{+}$. Second, if $%
u^{\prime }(\alpha )\in R_{0}$, then so does $u(\alpha )$. Let $\delta \in
R_{0,+}\cap S$. Since $u$ is of minimal length, we have $\ell (s_{\delta
}u)>\ell (u)$ (here $\ell $ is the length function) and this implies that $%
u^{-1}(\delta )\in R_{+}$ (see for example \cite[\S 1.6]{Hum}). It follows
that $u^{-1}(\beta )$ is positive for all $\beta \in R_{0,+}$. Therefore we
cannot have $u(\alpha )=-\beta \in R_{0,-}$ with $\beta \in R_{0,+}$,
since this would imply that $u^{-1}(\beta )=-\alpha \in R_{-}$. We have
shown that $u(\alpha )\in R_{+}$ in both cases, that is $u\in U$ as
required. \newline

We prove 2. By definition of $U$ we have $\overline{R}_{+}\subset
\bigcap\limits_{u\in U}u^{-1}(R_{+})$. Assume $\alpha\in\bigcap\limits_{u\in
U}u^{-1}(R_{+})$. We then have $u(\alpha)\in R_{+}$ for any $u\in U$.\
Consider $\gamma\in\overline{P}_{+}.$ By assertion 1, there exists $u\in U$
such that $\gamma\in u^{-1}(P_{+})$. We thus have $(\gamma,\alpha^{\vee
})=(u(\gamma),u(\alpha)^{\vee})\geq0$ for any $\gamma\in\overline{P}_{+} $.\
This implies that $\alpha$ is a positive root of $\overline{R}_{+}$.\newline

We prove 3. Recall that the stabilizer of $\rho$ under $W$ reduces to $\{1\} 
$. Consider $w\in W$. There exists $\overline{w}\in\overline{W}$ such that $%
\overline{w}(w^{-1}\cdot\rho)\in\overline{P}_{+}$. By assertion 1, there
exists $u\in U$ such that $u\overline{w}(w^{-1}\cdot\rho)\in P_{+}$. Since $%
\rho$ is the unique element of the orbit $W\cdot\rho$ in $P_{+},$ we must
have $w=u\overline{w}$. Now assume that there exist $u_{1},u_{2}\in U $ and $%
\overline{w}_{1},\overline{w}_{2}\in\overline{W}$ such that $u_{1}\overline{w%
}_{1}=u_{2}\overline{w}_{2}.$ We have $u_{2}=u_{1}\overline{w}$ with $%
\overline{w}=\overline{w}_{1}\overline{w}_{2}^{-1}\in\overline{W}$. If $%
\overline{w}\neq1$, there exists $\alpha\in\overline{R}_{+}$ such that $%
\overline{w}(\alpha)=-\beta$ with $\beta\in\overline{R}_{+}$. Then $%
(\rho,u_{2}(\alpha)^{\vee})=-(\rho,u_{1}(\beta)^{\vee})<0$ since $%
u_{1}(\beta)\in R_{+}$. This contradicts the hypothesis $u_{2}(\alpha)\in
R_{+}$. hence $\overline{w}=1,$ that is $\overline{w}_{1}=\overline{w}_{2}$
and $u_{1}=u_{2}$.
\end{proof}

Denote by $\overline{E}$ the $\mathbb{Q}$-vector space generated by the
roots in $\overline{R}_{+}$. Then we have $\overline{E}\cap R_{+}=\overline{R%
}_{+}$; see \cite[\S 1.10]{Hum}. We will make frequent use of this fact in
the rest of the paper. It is important to notice that this holds because we
assumed that $\overline{S}\subset S$.

\begin{Lem}
\label{LemU2} Let $u\in U$. Then $u(\overline{\rho})=\overline{\rho}$ if and
only if $u(\overline{R}_{+})=\overline{R}_{+}$.
\end{Lem}

\begin{proof}
Assume that there exists $\alpha\in\overline{R}_{+}$ such that $u(\alpha
)\notin\overline{R}_{+}$. Then since $u(\alpha)\in R_{+}$ we have $%
u(\alpha)\notin\overline{E}$. It follows that there exists a simple root $%
\alpha_{j}\notin\overline{R}_{+}$ such that $u(\alpha)\geq\alpha_{j}$. As $u(%
\overline{R}_{+})\subset R_{+}$, there can't be any cancellation of simple
roots when decomposing $u(\overline{\rho})$ on the basis $S$. Therefore we
have $u(\overline{\rho})\geq\alpha_{j}$ and $u(\overline{\rho})\notin 
\overline{E}$. From there, we see that we cannot have $u(\overline
\rho)=\overline\rho$ since $\overline\rho\in\overline{E}$. The converse is
trivial.
\end{proof}

\begin{Lem}
\label{LemU5} Let $u\in U$ be such that $u(\overline{\rho})\neq\overline{%
\rho }$. Then $u(\overline{\rho})\nless\overline{\rho}$.
\end{Lem}

\begin{proof}
Since $u(\overline{R}_{+})\neq \overline{R}_{+}$, arguing as in the proof of
the previous lemma, we know that there exists $\alpha \in \overline{R}_{+}$
and a simple root $\alpha _{j}\notin \overline{E}$ such that $u(\alpha )\geq
\alpha _{j}$. Since $\overline{\rho }\in \overline{E}$, the root $\alpha _{j}
$ appears in the decomposition of $u(\overline{\rho })-\overline{\rho }$ in
the basis $S$ with a positive coefficient hence we cannot have $u(\overline{%
\rho })<\overline{\rho }$.
\end{proof}

\begin{Lem}
\label{LemU3} Let $\gamma,\gamma^{\prime}\in P$ be such that $\gamma \leq_{%
\overline{R}_{+}}\gamma^{\prime}$. Then we have $u(\gamma)\leq_{R_{+}}u(%
\gamma^{\prime})$ for all $u\in U$.
\end{Lem}

\begin{proof}
By definition $\gamma\geq_{\overline{R}_{+}}\gamma^{\prime}$ implies that $%
\gamma-\gamma^{\prime}$ is a sum of roots in $\overline{R}_{+}$. Since $u(%
\overline{R}_{+})\subset R_{+}$ we see that $u(\gamma-\gamma^{\prime})$ is a
sum of roots in $R_{+}$. Hence $u(\gamma-\gamma^{\prime})=u(\gamma
)-u(\gamma^{\prime})\geq_{R_{+}}0$ as required.
\end{proof}


\begin{Lem}
\label{LemU4} Let $\gamma\in P$ be such that $\gamma\notin\overline{P}_{+}$.
Then we have $u(\gamma)\notin P_{+}$ for all $u\in U$.
\end{Lem}

\begin{proof}
Since $\gamma\notin\overline{P}_{+}$, there exists $\alpha\in\overline{R}%
_{+} $ such that $(\gamma,\alpha^{\vee})<0$. It follows that 
\begin{equation*}
(u(\gamma),u(\alpha)^{\vee})=(\gamma,\alpha^{\vee})<0. 
\end{equation*}
Since $u(\alpha)\in R_{+}$, this implies that $u(\gamma)\notin P_{+}$.
\end{proof}


\section{Induced characters}

\subsection{The functions $H_{\protect\mu}$}

Given $\mu \in \overline{P}_{+}$, write $H_{\mu }:=\mathrm{char}(V(\mu
)\uparrow _{\overline{\mathfrak{g}}}^{\mathfrak{g}})$ the induced character
of $\overline{V}(\mu )$ from $\overline{\mathfrak{g}}$ to $\mathfrak{g}$. We
then have 
\begin{equation*}
H_{\mu }:=\sum_{\lambda \in P_{+}}m_{\mu }^{\lambda }s_{\lambda }.
\end{equation*}%
Observe there can exist infinitely many weights $\lambda $ such that $m_{\mu
}^{\lambda }\neq 0$. When $\overline{\mathfrak{g}}=\mathfrak{h}$ is reduced
to the Cartan subalgebra, we have $\overline{R}_{+}=\emptyset $ and we set $%
m_{\lambda }^{\mu }=K_{\lambda ,\mu }=\dim V(\lambda )_{\mu }$ so that 
\begin{equation}
h_{\mu }:=\sum_{\lambda \in P_{+}}K_{\lambda ,\mu }s_{\lambda }.
\label{dech}
\end{equation}%
Since $K_{\lambda ,\mu }=K_{\lambda ,w(\mu )}$ for any $w\in W$, we have $%
h_{\mu }=h_{w(\mu )}$ (for the usual action of $W$ on $P$). Moreover, $%
K_{\mu ,\mu }=1$ and $K_{\lambda ,\mu }\neq 0$ if and only if $\lambda \geq
\mu $ (i.e. $\lambda -\mu $ decomposes as a sum of simple roots). The sets $%
\{s_{\lambda }\mid \lambda \in P_{+}\}$ and $\{h_{\lambda }\mid \lambda \in
P_{+}\}$ are bases of $\mathbb{G}$ and the corresponding transition matrix
is unitriangular for the order $\leq $.

We now define two $\mathbb{Z}$-linear maps $H$ and $S$ by 
\begin{equation*}
H:\left\{ 
\begin{array}{c}
\mathbb{Z}[P]\rightarrow\mathbb{G} \\ 
e^{\beta}\mapsto h_{\beta}%
\end{array}
\right. \text{ and }S:\left\{ 
\begin{array}{c}
\mathbb{Z}[P]\rightarrow\mathbb{G} \\ 
e^{\beta}\mapsto s_{\beta}%
\end{array}
\right. 
\end{equation*}
and we set%
\begin{equation*}
\Delta=\prod_{\alpha\in R_{+}}(1-e^{\alpha}). 
\end{equation*}

\begin{Prop}
\label{prop_SH}The maps $H$ and $S$ satisfy the relations 
\begin{equation*}
S(e^{\beta})=H(\Delta e^{\beta})\text{ and }H(e^{\beta})=S(\Delta^{-1}e^{%
\beta}) 
\end{equation*}
for any $\beta\in P$. Therefore $S=H\circ\Delta$ and $H=S\circ\Delta^{-1}$
(by writing for short $\Delta$ and $\Delta^{-1}$ for the multiplication by $%
\Delta$ and $\Delta^{-1}$ in $\mathbb{Z}[[P]]$).
\end{Prop}

\begin{proof}
The partition function $\mathcal{P}$ is defined by 
\begin{equation*}
\Delta^{-1}=\prod_{\alpha\in R_{+}}\frac{1}{1-e^{\alpha}}=\sum_{\gamma\in P}%
\mathcal{P}(\gamma)e^{\gamma}
\end{equation*}
and we have by definition $h_{\beta}=\sum_{\lambda}K_{\lambda,\beta}s_{%
\lambda}$ where $K_{\lambda,\beta}=\sum_{w}\varepsilon(w)\mathcal{P}%
(w\circ\lambda-\beta)$.\ This gives 
\begin{equation*}
S(\Delta^{-1}e^{\beta})=\sum_{\gamma\in P}\mathcal{P}(\gamma)s_{\beta+\gamma
}. 
\end{equation*}
Let $\gamma\in P$. Then either $s_{\beta+\gamma}=0$ or there exists $%
\lambda\in P_{+}$ and $w\in W$ such that $w^{-1}\circ(\beta+\gamma)=\lambda$%
, that is $\gamma=w\circ\lambda-\beta$. This yields $s_{\beta+\gamma
}=\varepsilon(w)s_{\lambda}$ and in turn we obtain 
\begin{equation*}
S(\Delta^{-1}e^{\beta})=\sum_{\lambda}\sum_{w\in W}\varepsilon(w)\mathcal{P}%
(w\circ\lambda-\beta)s_{\lambda}=\sum_{\lambda}K_{\lambda,\beta}s_{\lambda
}=h_{\beta}
\end{equation*}
as desired.\ Note that we have for any $U\in\mathbb{Z}[P]$, $H(U):=S(\Delta
^{-1}U) $. Then if we set $U=\Delta e^{\beta}$, we get the relation $%
H(\Delta e^{\beta})=S(e^{\beta})$, as required.
\end{proof}

\bigskip

Now write $\overline{\mathbb{G}}=\mathbb{Z}^{\overline{W}}[e^{\beta }\mid
\beta \in P]$ the character ring of $\overline{\mathfrak{g}}$ (polynomials
of $\mathbb{Z}[P]$ invariant under the action of $\overline{W}$ the Weyl
group of $\overline{\mathfrak{g}}$). The set of irreducible characters $\{%
\overline{s}_{\mu }\mid \mu \in \overline{P}_{+}\}$ of $\overline{\mathfrak{g%
}}$ is a basis of $\overline{\mathbb{G}}$. Define the $\mathbb{Z}$-linear
map 
\begin{equation*}
\overline{H}:\left\{ 
\begin{array}{c}
\mathbb{Z}[P]\rightarrow \mathbb{G} \\ 
e^{\mu }\mapsto H_{\mu }%
\end{array}%
\right. 
\end{equation*}%
and set 
\begin{equation*}
\overline{\Delta }=\prod_{\alpha \in R_{+}\setminus \overline{R}%
_{+}}(1-e^{\alpha })\text{ and }\overline{\bigtriangledown }=\prod_{\alpha
\in \overline{R}_{+}}(1-e^{\alpha }).
\end{equation*}

\begin{Prop}
\ \ 

\begin{enumerate}
\item The maps $\overline{H}$ and $S$ satisfy the relation 
\begin{equation*}
\overline{H}(e^{\mu})=S(\overline{\Delta}^{-1}e^{\mu}) 
\end{equation*}
for any $\mu\in P$. We write for short $\overline{H}=S\circ\overline{\Delta }%
^{-1}$.

\item We have $\overline{H}(e^{\mu})=H(\overline{\bigtriangledown}e^{\mu})$.
\end{enumerate}
\end{Prop}

\begin{proof}
The first assertion is proved as in the previous proof by replacing the
partition function $\mathcal{P}$ by $\overline{\mathcal{P}}$. For the second
one, we combine the first part with the previous proposition.
\end{proof}

\bigskip

We have, using the Weyl character formula: 
\begin{equation*}
\overline{\bigtriangledown}=\prod_{\alpha\in\overline{R}_{+}}(1-e^{\alpha
})=\sum_{\overline{w}\in\overline{W}}\varepsilon(\overline{w})e^{\overline {%
\rho}-\overline{w}(\overline{\rho})}
\end{equation*}
where $\overline{\rho}$ is the half sum of positive roots of $\overline {%
\mathfrak{g}}$. By the second assertion of the previous proposition, we get
for all $\mu\in P$%
\begin{equation*}
H_{\mu}=\overline{H}(e^{\mu})=\sum_{\overline{w}\in\overline{W}}\varepsilon(%
\overline{w})h_{\mu+\overline{\rho}-\overline{w}(\overline{\rho})}. 
\end{equation*}


\subsection{Irreducible components of $\overline{R}$}

\label{subsec_Dec}Now assume the semisimple Lie algebra $\overline{\mathfrak{%
g}}$ has a decomposition of the form 
\begin{equation*}
\overline{\mathfrak{g}}\mathfrak{=g}_{1}\oplus \mathfrak{g}_{2}\oplus \cdots
\oplus \mathfrak{g}_{r}
\end{equation*}%
where each $\mathfrak{g}_{k},$ $k=1,\ldots ,r$ is a Lie subalgebra of $%
\mathfrak{g}$ with irreducible root system $R_{k}\subset \overline{R}$ and $%
\overline{R}=\bigsqcup\limits_{k=1}^{r}R^{(k)}$.\ We also assume that we
have $P=P^{(1)}\oplus \cdots \oplus P^{(r)}$ where $P^{(k)}$ is the weight
lattice of $\mathfrak{g}_{k}$.\ In particular each weight $\mu \in P_{+}$
decomposes on the form $\mu =\mu ^{(1)}+\cdots +\mu ^{(r)}$ with $\mu
^{(k)}\in P_{+}^{(k)}$. We then have additional properties for the
functions $H_{\mu }$ we shall need in Section \ref{Sec_class}. \\

We have 
\begin{equation*}
\overline{\bigtriangledown }=\prod_{\alpha \in \overline{R}_{+}}(1-e^{\alpha
})=\prod_{k=1}^{r}\prod_{\alpha \in R_{+}^{(k)}}(1-e^{\alpha })
\end{equation*}%
and%
\begin{equation*}
H_{\mu }=\prod_{k=1}^{r}\prod_{\alpha \in R_{+}^{(k)}}(1-e^{\alpha })h_{\mu
^{(1)}+\cdots +\mu ^{(r)}}.
\end{equation*}%
Combining (\ref{dech}) and Proposition \ref{prop_SH} (for each root system $%
R_{k})$, we get for any $k=1,\ldots ,r,$%
\begin{equation*}
\prod_{\alpha \in R_{+}^{(k)}}(1-e^{\alpha })h_{\mu ^{(1)}+\cdots +\mu
^{(r)}}=\sum_{\lambda ^{(k)}\in P_{+}^{(k)}}K_{\lambda ^{(k)},\mu
^{(k)}}^{-1}h_{\mu ^{(1)}+\cdots \lambda ^{(k)}+\cdots +\mu ^{(r)}}
\end{equation*}%
where the coefficients $K_{\lambda ^{(k)},\mu ^{(k)}}^{-1}$ are those of the
inverse matrix of $(K_{\lambda ^{(k)},\mu ^{(k)}})_{\lambda ^{(k)},\mu
^{(k)}\in P_{+}^{(k)}}$. By an easy induction, we then get 
\begin{equation}
H_{\mu }=\sum_{\lambda ^{(1)}\in P_{+}^{(k)}}\cdots \sum_{\lambda ^{(r)}\in
P_{+}^{(r)}}K_{\lambda ^{(1)},\mu ^{(1)}}^{-1}\cdots K_{\lambda ^{(r)},\mu
^{(r)}}^{-1}h_{\lambda ^{(1)}+\cdots +\lambda ^{(r)}}.  \label{H-product}
\end{equation}

\subsection{The conjecture}

We start with an easy observation.

\begin{Lem}
Consider $u\in W$. Then the two following statements are equivalent :

\begin{enumerate}
\item $u(\overline{R}_{+})=\overline{R}_{+}$

\item $u$ is a Dynkin diagram automorphism of $\overline{\mathfrak{g}}$
\end{enumerate}
\end{Lem}

\begin{proof}
When $u$ is a Dynkin diagram automorphism of $\overline{\mathfrak{g}}$, we
clearly have $u(\overline{R}_{+})=\overline{R}_{+}$.\ Now assume $u(%
\overline{R}_{+})=\overline{R}_{+}$. Then we have $u(\overline {R})=%
\overline{R}$ and $u$ is an automorphism of the root system $\overline {R}$%
.\ It is known that $\mathrm{Aut}(\overline{R})=\overline{W}\ltimes\mathrm{%
Aut}(\overline{\Gamma})$ where $\overline{\Gamma}$ is the Dynkin diagram of $%
\overline{R}$ i.e. $\mathrm{Aut}(\overline{R})$ is the semidirect product of 
$\overline{W}$ (which is normal in $\mathrm{Aut}(\overline{R})$) with $%
\mathrm{Aut}(\overline{\Gamma})$.\ Since $u(\overline {R}_{+})=\overline{R}%
_{+}$ the element $u$ belongs in fact in $\mathrm{Aut}(\overline{\Gamma})$
(otherwise $u$ would send at least a positive root of $\overline{R}_{+}$ on
a negative root).
\end{proof}

\bigskip

\begin{Prop}
\label{Prop_direct} Let $\mu ,\nu \in \overline{P}_{+}$. Assume that there
exists $u\in W$ such that $u(\overline{R}_{+})=\overline{R}_{+}$ and $\nu
=u(\mu )$ (or equivalently, $\mu $ and $\nu $ are conjugate by a Dynkin
diagram automorphism of $\overline{\mathfrak{g}}$ lying in the Weyl group of 
$\mathfrak{g}$). Then $H_{\mu }=H_{\nu }$.
\end{Prop}

\begin{proof}
With the previous notation, we have%
\begin{equation*}
H_{\mu}=H\bigg(\prod_{\alpha\in\overline{R}_{+}}(1-e^{\alpha})e^{\mu }\bigg)%
\text{ and } H_{\nu}=H\bigg(\prod_{\alpha\in\overline{R}_{+}}(1-e^{%
\alpha})e^{\nu}\bigg). 
\end{equation*}
Since $u(\overline{R}_{+})=\overline{R}_{+}$ we see that $u(\overline
\rho)=\overline\rho$ and that $u\overline{W}u^{-1}=\overline{W}$ (indeed, $%
us_{\alpha}u^{-1}=s_{u\alpha}$ for all $\alpha\in\overline{R}$). Therefore 
\begin{equation*}
\prod_{\alpha\in\overline{R}_{+}}(1-e^{\alpha})e^{\nu}=\sum_{w\in\overline{W}%
}\varepsilon(w)e^{\nu+\overline{\rho}-w(\overline{\rho})}=\sum_{w\in 
\overline{W}}\varepsilon(w)e^{u(\mu)+u(\overline{\rho})-uw(u^{-1}(\overline{%
\rho}))}=\sum_{w\in\overline{W}}\varepsilon(w)e^{u(\mu +\overline{\rho}-w(%
\overline{\rho}))}. 
\end{equation*}
It follows that 
\begin{equation*}
H_{\nu}=H\bigg(\sum_{w\in\overline{W}}\varepsilon(w)e^{u(\mu+\overline{\rho }%
-w(\overline{\rho}))}\bigg) =\sum_{w\in\overline{W}}\varepsilon (w)h_{u(\mu+%
\overline{\rho}-w(\overline{\rho}))}=\sum_{w\in\overline{W}%
}\varepsilon(w)h_{\mu+\overline{\rho}-w(\overline{\rho})}=H_{\mu}
\end{equation*}
since $h_{w(\beta)}=h_{\beta}$ for any $w\in W$.
\end{proof}

\bigskip

We conjecture that the converse is true:

\begin{conj}
\label{conj}Consider $\mu,\nu\in\overline{P}_{+}$. Then we have $H_{\mu
}=H_{\nu}$ if and only if there exists $u$ in $W$ such that $u(\overline {R}%
_{+})=\overline{R}_{+}$ and $\nu=u(\mu)$ or equivalently, $\mu$ and $\nu$
are conjugate by a Dynkin diagram automorphism of $\overline{\mathfrak{g}}$
lying in the Weyl group of $\mathfrak{g}$.
\end{conj}


\section{Triangular decomposition of $H_{\protect\mu}$}


\subsection{Decomposition on the $h$-basis}

Let $\mu\in\overline{P}_{+}$ and let $w\in U$ be such that $\mu\in\overline {%
C}_{w}=w^{-1}\overline{C}$. Recall that $R_{+}^{w}=w^{-1}(R_{+})$.\ Since $%
w\in U$, we have that $w(\overline{R}_{+})\subset R_{+}$ which in turn
implies that $\overline{R}_{+}\subset R_{+}^{w}$, that is $\preceq\subset
\leq_{w}.$

\begin{Prop}
\label{Prop_Triang} Let $w\in U$. We have for all $\mu\in\overline{P}_{+}$%
\begin{equation*}
H_{\mu}=h_{\mu}+\sum_{\lambda\in P_{+}^{w},\mu<_{w}\lambda}a_{\lambda,\mu
}h_{\lambda}
\end{equation*}
where for any $\lambda\in P_{+}^{w}$%
\begin{equation*}
a_{\lambda,\mu}=\sum_{\overline{w}\in\overline{W}\mid\mu+\overline{\rho }-%
\overline{w}(\overline{\rho})\in W\cdot\lambda}\varepsilon(\overline{w}). 
\end{equation*}
\end{Prop}

\begin{proof}
Since $\preceq \subset \leq _{w},$ we have%
\begin{equation*}
H_{\mu }=h_{\mu }+\sum_{\overline{w}\in \overline{W},\overline{w}\neq
1}\varepsilon (\overline{w})h_{\mu +\overline{\rho }-\overline{w}(\overline{%
\rho })}\text{ with }\mu <_{w}\mu +\overline{\rho }-\overline{w}(\overline{%
\rho })\text{ for }\overline{w}\neq 1.
\end{equation*}%
Now for each $w\neq 1$, the orbit of each $\gamma =\mu +\overline{\rho }-%
\overline{w}(\overline{\rho })$ intersects $P_{+}^{w}$ at one point (say $%
\lambda $) and we can use the relations $h_{w(\gamma )}=h_{\gamma }$ for any 
$w\in W$. Moreover, we then have $\gamma \leqslant _{w}\lambda $. We thus
obtain $\mu <_{w}\mu +\overline{\rho }-\overline{w}(\overline{\rho }%
)\leqslant _{w}\lambda $ which gives the unitriangularity of the
decomposition. The coefficients $a_{\lambda ,\mu }$ are then obtained by
gathering the contributions in $h_{\lambda }$ for each $\lambda \in P_{+}^{w}
$.
\end{proof}

\begin{Rem}
\begin{enumerate}
\item For $\mathfrak{g=}\overline{\mathfrak{g}}$, the coefficients $%
a_{\lambda ,\mu }$ are the entries of the inverse matrix $K^{-1}$ where $%
K=(K_{\lambda ,\mu })_{\lambda ,\mu \in P_{+}}$.\ In type $A$, $K$ is the
Kostka matrix.\ Obtaining a combinatorial formula for the coefficients of $%
K^{-1}$ is already a nontrivial problem (see \cite{Dua} and the references
therein). \ As far as we are aware no such description for the coefficients
of $K^{-1}$ exists for other root systems (and thus also for the
coefficients $a_{\lambda ,\mu }$ associated to a general Levi subalgebra).

\item We can also deduce from Propositions \ref{Lem-U1} and \ref{Prop_Triang}
that for any $u\in U$, the set $\{H_{\lambda}\mid\lambda\in P_{+}^{u}\}$ is
a basis of $\mathbb{G}$.
\end{enumerate}
\end{Rem}


\subsection{Consequences}

\begin{Prop}
\label{Prop_memeorbit} Let $\mu$ and $\nu$ be dominant weights in $\overline{%
P}_{+}$ such that $H_{\mu}=H_{\nu}$. Then, there exists $\tau\in W$ such
that $\tau(\nu)=\mu$. In particular, if $\mu$ and $\nu$ belong to the same
closed Weyl chamber for $\mathfrak{g}$, we have $\tau=1$ and $\mu=\nu$.
\end{Prop}

\begin{proof}
Assume that $\mu $ belongs to $\overline{P}_{+}^{w}$ and $\nu $ belongs to $%
\overline{P}_{+}^{w^{\prime }}$ with $w,w^{\prime }$ in $U$. Let $\tau \in W$
be such that $w^{\prime }=w\tau $. 
We then have $R_{+}^{w^{\prime }}=\tau ^{-1}(R_{+}^{w})$ and $%
P_{+}^{w^{\prime }}=\tau ^{-1}(P_{+}^{w})$. Moreover $\mu <_{w}\gamma $ if
and only if $\tau ^{-1}(\mu )<_{w^{\prime }}\tau ^{-1}(\gamma )$. On the one
hand, using Proposition \ref{Prop_Triang}, we get%
\begin{align*}
H_{\nu }& =h_{\nu }+\sum_{\lambda \in P_{+}^{w^{\prime }},\nu <_{w^{\prime
}}\lambda }a_{\lambda ,\nu }h_{\lambda } \\
& =h_{\nu }+\sum_{\lambda \in P_{+}^{w},\tau (\nu )<_{w}\lambda }a_{\tau
^{-1}(\lambda ),\nu }h_{\tau ^{-1}(\lambda )}.
\end{align*}%
Since $h_{w(\beta )}=h_{\beta }$ for all $w\in W$ and $\beta \in P$, this
can be rewritten under the form%
\begin{equation*}
H_{\nu }=h_{\tau (\nu )}+\sum_{\lambda \in P_{+}^{w},\tau (\nu )<_{w}\lambda
}a_{\tau ^{-1}(\lambda ),\nu }h_{\lambda }.
\end{equation*}%
On the other hand we have 
\begin{equation*}
H_{\mu }=h_{\mu }+\sum_{\lambda \in P_{+}^{w},\mu <_{w}\lambda }a_{\lambda
,\mu }h_{\lambda }.
\end{equation*}%
So $H_{\nu }=H_{\mu }$ implies that $h_{\tau (\nu )}=h_{\mu }$ by comparing
the indices of the basis vectors of $\{h_{\lambda }\mid \lambda \in
P_{+}^{w}\}$ which are minimal for the order $\leq _{w}$. Hence $\mu =\tau
(\nu )$ as desired.
\end{proof}

\medskip

\begin{Rem}
If $H_{\mu}=H_{0}$ (i.e. we have $\nu=0$), then $\mu=0$ since $\mu$ et $0 $
always belong to the same closed Weyl chamber.
\end{Rem}

\medskip

For any weight $\mu \in \overline{P}_{+},$ define the set $E_{\mu }=\{\mu +%
\overline{\rho }-\overline{w}(\overline{\rho })\mid \overline{w}\in 
\overline{W}\}$. Since the stabilizer of $\overline{\rho }$ under the action
of $\overline{W}$ reduces to $\{1\}$, the cardinality of $E_{\mu }$ is equal
to that of $\overline{W}$. The following corollary shows that the conjecture
holds when each of the sets $E_{\mu }$ and $E_{\nu }$ is contained in a
closed Weyl chamber. This happens in particular when $\mu $ and $\nu $ are
sufficiently far from the walls of the Weyl chambers in which they appear.

\begin{Cor}
Let $\mu$ and $\nu$ be two dominant weights in $\overline{P}_{+}$. Assume
that there exist $w\in W$ such that $E_{\mu}\subset P_{+}^{w}$ and $%
w^{\prime}\in W$ such that $E_{\nu}\subset P_{+}^{w^{\prime}}$.\ Then $%
H_{\mu}=H_{\nu}$ implies that $\nu=\tau(\mu)$ and $\tau(\overline{R}_{+})=%
\overline{R}_{+}$ with $\tau\in W$ such that $w^{\prime}=w\tau$.
\end{Cor}

\begin{proof}
All the elements of $E_{\mu}$ belong to $P_{+}^{w}$. They thus belong to
distinct $W$-orbits.\ Hence the decomposition of $H_{\mu}$ in the basis $%
\{h_{\lambda}\mid\lambda\in P_{+}^{w}\}$ is%
\begin{equation*}
H_{\mu}=h_{\mu}+\sum_{\overline{w}\in\overline{W},\overline{w}\neq
1}\varepsilon(\overline{w})h_{\mu+\overline{\rho}-\overline{w}(\overline{%
\rho })}.   \label{Hmu}
\end{equation*}
Similarly, the elements of $E_{\nu}$ belong to distinct $W$-orbits. Hence
the decomposition of $H_{\nu}$ in the basis $\{h_{\lambda}\mid\lambda\in
P_{+}^{w^{\prime}}\}$ is%
\begin{equation*}
H_{\nu}=h_{\nu}+\sum_{\overline{w}^{\prime}\in\overline{W},\overline {w}%
^{\prime}\neq1}\varepsilon(\overline{w}^{\prime})h_{\nu+\overline{\rho }-%
\overline{w}^{\prime}(\overline{\rho})}. 
\end{equation*}
Since $H_{\nu}=H_{\mu}$, we see that there exists $\tau\in W$ such that $%
\tau(\nu)=\mu$ by the previous proposition. Further, we know that $\tau$ is
such that $P_{+}^{w^{\prime}}=\tau^{-1}(P_{+}^{w})$ thus we have $%
\tau(E_{\nu })=E_{\mu}$. Let $\alpha\in\overline{R}_{+}$ and $\overline{w}%
=s_{\alpha}$. Then $\overline{w}(\overline{\rho})-\overline{\rho}=\alpha$
and we see that there exists an element $\overline{w}^{\prime}\in\overline{W}
$ such that $\tau(\nu+\alpha)=\mu+\overline{\rho}-\overline{w}^{\prime}(%
\overline{\rho})$. In turn, this implies $\tau(\alpha)=\overline{\rho}-%
\overline{w}^{\prime }(\overline{\rho})$ as $\tau(\nu)=\mu$ and $\tau(\alpha)
$ is a sum of positive roots in $\overline{R}_{+}$. But $\tau(\alpha)$ also
lies in $R$, hence $\tau(\alpha)\in\overline{R}_{+}$; see Section \ref%
{section3}. We have shown that $\tau$ maps $\overline{R}_{+}$ onto itself as
expected.
\end{proof}


\section{The functions $M_{\protect\mu}$}

We know give an equivalent formulation of our problem in terms of parabolic
analogues of monomial functions.

\subsection{Decomposition on the monomial functions}

For any weight $\gamma\in P$, set $\mathrm{m}_{\gamma}=\sum_{w\in
W}e^{w(\gamma)}$ so that $\mathrm{m}_{\gamma}$\footnote{%
Our function $\mathrm{m}_{\gamma}$ slightly differs from the usual monomial
function $m_{\gamma}=\frac{1}{|W_{\gamma}|}\sum_{w\in W}e^{w(\gamma)}$ where 
$W_{\gamma}$ is the stabilizer of $\gamma$ under the action of $W$.} is the
image of $e^{\gamma}$ by the symmetrization operator 
\begin{equation*}
\mathcal{M}:\left\{ 
\begin{array}{c}
\mathbb{Z}[P]\rightarrow\mathbb{Z}[P]^{W} \\ 
e^{\gamma}\mapsto\mathrm{m}_{\gamma}%
\end{array}
\right. . 
\end{equation*}
We clearly have $\mathrm{m}_{w(\gamma)}=\mathrm{m}_{\gamma}$ for any $w\in W$%
. Also $\{\mathrm{m}_{\lambda}\mid\lambda\in P_{+}^{w}\}$ is a basis of $%
\mathbb{G}$.\ Given $\mu\in P$, set%
\begin{equation*}
M_{\mu}:=\mathcal{M}(\prod_{\alpha\in\overline{R}_{+}}(1-e^{\alpha})e^{\mu
})=\sum_{\overline{w}\in\overline{W}}\varepsilon(\overline{w})\mathrm{m}%
_{\mu+\overline{\rho}-\overline{w}(\overline{\rho})}. 
\end{equation*}

\begin{Lem}
\ \label{Lemme_M}

\begin{enumerate}
\item We have 
\begin{equation*}
M_{\mu}=\sum_{\lambda\in P_{+}}a_{\lambda,\mu}\mathrm{m}_{\lambda}\text{
with }a_{\lambda,\mu}=\sum_{\overline{w}\in\overline{W}\mid\mu+\overline{%
\rho }-\overline{w}(\overline{\rho})\in W\cdot\lambda}\varepsilon(\overline{w%
}).   \label{M}
\end{equation*}

\item Consider $\mu,\nu\in\overline{P}_{+}$. Then $H_{\mu}=H_{\nu}$ if and
only if $M_{\mu}=M_{\nu}$.
\end{enumerate}
\end{Lem}

\begin{proof}
Assertion 1 follows from the identity $\mathrm{m}_{w(\gamma)}=\mathrm{m}%
_{\gamma}$ for any $\gamma\in P$ and any $w\in W$.\ By Proposition \ref%
{Prop_Triang}, the coefficients of the expansion of $M_{\mu}$ on the basis $%
\{\mathrm{m}_{\lambda}\mid\lambda\in P_{+}\}$ are the same as those
appearing in the expansion of $H_{\mu}$ on the basis $\{h_{\lambda}\mid%
\lambda\in P_{+}^{w}\}$.\ Assertion 2 follows.
\end{proof}

\subsection{A simple expression for the functions $M_{\protect\lambda}$}

For any $\gamma\in P$, set 
\begin{equation*}
\overline{a}_{\gamma}=\sum_{\overline{w}\in\overline{W}}\varepsilon (%
\overline{w})e^{\overline{w}(\gamma)}. 
\end{equation*}
We thus have $\overline{a}_{\overline{w}(\gamma)}=\varepsilon(\overline {w})%
\overline{a}_{\gamma}$ and $\overline{w}(\overline{a}_{\gamma })=\varepsilon(%
\overline{w})\overline{a}_{\gamma}$ for any $\overline{w}\in\overline{W}$
and $\overline{a}_{\overline{w}_{0}(\overline{\rho})}=\varepsilon(\overline{w%
}_{0})\overline{a}_{\overline{\rho}}$ where $\overline{w}_{0}$ is the
element of maximal length in $\overline{W}$. 


\begin{Prop}
Let $\mu\in\overline{P}_{+}$.
\begin{enumerate}
\item We have 
\begin{equation*}
M_{\mu}=\varepsilon(\overline{w}_{0})\sum_{u\in U}u(\overline{a}_{\mu+%
\overline{\rho}}\cdot\overline{a}_{\overline{\rho}}).   \label{Mmu}
\end{equation*}

\item Let $\Lambda$ be the unique element lying in $\{u(\mu+2\overline{\rho }%
)\mid u\in U\}\cap P_{+}$. Then we have 
\begin{equation*}
M_{\mu}=\varepsilon(\overline{w}_{0})e^{\Lambda}+\sum_{\gamma\in
P,\gamma<\Lambda}b_{\lambda,\mu}e^{\gamma}. 
\end{equation*}
\end{enumerate}
\end{Prop}

\begin{proof}
We prove (1). We have 
\begin{equation*}
M_{\mu}=\sum_{\overline{w}\in\overline{W}}\varepsilon(\overline{w})\mathrm{m}%
_{\mu+\overline{\rho}-\overline{w}(\overline{\rho})}=\sum_{w\in W}w\left(
e^{\mu+\overline{\rho}}\sum_{\overline{w}\in\overline{W}}\varepsilon(%
\overline{w})e^{-\overline{w}(\overline{\rho})}\right) . 
\end{equation*}
This gives 
\begin{equation*}
M_{\mu}=\sum_{w\in W}w\left( e^{\mu+\overline{\rho}}\overline{a}_{-\overline{%
\rho}}\right) =\varepsilon(\overline{w}_{0})\sum_{w\in W}w\left( e^{\mu+%
\overline{\rho}}\overline{a}_{\overline{\rho}}\right) =\varepsilon(\overline{%
w}_{0})\sum_{u\in U}u\left( \sum_{\overline{w}\in\overline{W}}\overline{w}%
\left( e^{\mu+\overline{\rho}}\overline {a}_{\overline{\rho}}\right) \right) 
\end{equation*}
by using Assertion 3 of Proposition \ref{Lem-U1}. Hence%
\begin{align*}
M_{\mu} & =\varepsilon(\overline{w}_{0})\sum_{u\in U}u\left( \sum _{%
\overline{w}\in\overline{W}}e^{\overline{w}(\mu+\overline{\rho})}\overline{w}%
(\overline{a}_{\overline{\rho}})\right) \\
& =\varepsilon(\overline{w}_{0})\sum_{u\in U}u\left( \overline{a}_{\overline{%
\rho}}\sum_{\overline{w}\in\overline{W}}\varepsilon(\overline {w})e^{%
\overline{w}(\mu+\overline{\rho})}\right) \\
& =\varepsilon(\overline{w}_{0})\sum_{u\in U}u(\overline{a}_{\mu +\overline{%
\rho}}\cdot\overline{a}_{\overline{\rho}})
\end{align*}
since $\overline{a}_{\overline{w}(\overline{\rho})}=\varepsilon(\overline {w}%
)\overline{a}_{\overline{\rho}}.$\newline

We prove (2). The monomials $e^{\mu +\overline{\rho }}$ and $e^{\overline{%
\rho }}$ are the monomials of highest weight (with respect to~$\leq _{%
\overline{R}_{+}}$) appearing in the expression of $\overline{a}_{\mu +%
\overline{\rho }}$ and $\overline{a}_{\overline{\rho }}$ respectively. It
follows that the monomial $e^{\mu +2\overline{\rho }}$ is of highest weight
among those appearing in $\overline{a}_{\mu +\overline{\rho }}\cdot 
\overline{a}_{\overline{\rho }}$. Thus using (1) we get an expression of the
form 
\begin{equation*}
M_{\mu }=\varepsilon (\overline{w}_{0})\sum_{u\in U}u\left( e^{\mu +2%
\overline{\rho }}+\sum_{\nu <_{\overline{R}_{+}}\mu +2\overline{\rho }}%
\mathbb{Z}e^{\nu }\right) .
\end{equation*}%
By Lemma \ref{LemU3}, $\nu <_{\overline{R}_{+}}\mu +2\overline{\rho }$
implies that $u(\nu )<u(\mu +2\overline{\rho })$. Finally, the maximal
weight with respect to $\leq $ in the set $\{u(\mu +2\overline{\rho })\mid
u\in U\}$ is the unique element $\Lambda $ lying in $\{u(\mu +2\overline{%
\rho })\mid u\in U\}\cap P_{+}$. Therefore we have 
\begin{equation*}
M_{\mu }=\varepsilon (\overline{w}_{0})e^{\Lambda }+\sum_{\gamma \in
P,\gamma <\Lambda }b_{\lambda ,\mu }e^{\gamma }
\end{equation*}%
as required.
\end{proof}

\subsection{Proof of the conjecture for $\protect\mu+2\overline{\protect\rho}
$ dominant}

\begin{Lem}
Let $\mu\in\overline{P}_{+}$ be such that $\mu+2\overline{\rho}$ belongs to $%
P_{+}$. Then $\mu\in P_{+}.$
\end{Lem}

\begin{proof}
For any simple root $\alpha_{i}\in S$, we have $(\mu+2\overline{\rho}%
,\alpha_{i}^{\vee})\geq0$ since $\mu+2\overline{\rho}\in P_{+}$.\ Also for
any simple root $\alpha_{i}\in\overline{S}$, we have $(\mu,\alpha_{i}^{%
\vee})\geq0$ since $\mu\in\overline{P}_{+}$.\ Now consider $\alpha_{j}\in
S\setminus\overline{S}$.\ Since $2\overline{\rho}$ decomposes as a sum of
simple roots in $\overline{S}$, we must have $(2\overline{\rho},\alpha
_{j}^{\vee})\leq0$. Indeed for any $\alpha_{i}\in\overline{S}$, $(\alpha
_{i},\alpha_{j}^{\vee})=0$ or is negative as it can be easily seen by
considering the the entries of the Cartan matrix of $\mathfrak{g}$ which do
not appear on the diagonal. Therefore $(\mu,\alpha_{j}^{\vee})\geq (\mu+2%
\overline{\rho},\alpha_{j}^{\vee})\geq0$.
\end{proof}

\begin{Prop}
\label{Prop_mu+2ro} Let $\mu,\nu\in\overline{P}_{+}$ be such that $H_{\mu
}=H_{\nu}$ and assume that $\mu+2\overline{\rho}\in P_{+}$. Then, there
exists $v\in U$ such that $\nu=v(\mu)$ and $v(\overline{R}_{+})=\overline{R}%
_{+}$.
\end{Prop}

\begin{proof}
By the previous lemma, we see that $\mu\in P_{+}$. Let $v\in U$ be such that 
$\nu\in P_{+}^{v}$. Then by (the proof of) Proposition \ref{Prop_memeorbit},
we know that $v(\nu)=\mu$. Next Lemma \ref{Lemme_M} implies that $M_{\mu
}=M_{\nu}$ and, in particular, $M_{\mu}$ and $M_{\nu} $ have the same
maximal monomial with respect to $<$.\ Hence%
\begin{equation*}
\{u(\mu+2\overline{\rho})\mid u\in U\}\cap P_{+}=\{u(v^{-1}(\mu)+2\overline {%
\rho})\mid u\in U\}\cap P_{+}. 
\end{equation*}
But $\mu+2\overline{\rho}\in P_{+}$ so we have $\{u(v^{-1}(\mu)+2\overline {%
\rho})\mid u\in U\}\cap P_{+}=\{\mu+2\overline{\rho}\}$. Hence, there exists 
$u\in U$ such that $u(v^{-1}(\mu)+2\overline{\rho})=\mu+2\overline{\rho}$.
We have 
\begin{equation*}
\begin{array}{c}
\mu+2\overline{\rho}=u(v^{-1}(\mu)+2\overline{\rho}) \\ 
\Updownarrow \\ 
u^{-1}(\mu+2\overline{\rho})=v^{-1}(\mu)+2\overline{\rho} \\ 
\Updownarrow \\ 
vu^{-1}(\mu+2\overline{\rho})=\mu+2v(\overline{\rho}) \\ 
\Updownarrow \\ 
vu^{-1}(\mu+2\overline{\rho})-(\mu+2\overline{\rho})=2(v(\overline{\rho })-%
\overline{\rho}).%
\end{array}
\end{equation*}
Since $\mu+2\overline{\rho}\in P_{+}$, we have $vu^{-1}(\mu+2\overline{\rho }%
)-(\mu+2\overline{\rho})\leq0$. Hence $v(\overline{\rho})\leq\overline{\rho}$%
. By Lemma \ref{LemU5}, this implies that $v(\overline{\rho})=\overline{\rho 
}$. Finally by Lemma \ref{LemU2}, we have $v(\overline{R}_{+})=\overline {R}%
_{+}$.
\end{proof}

\begin{Rem}
We will see in the next section (Remark \ref{rem_Ce}) that we can have $\mu$
and $\nu$ in the same $W$-orbit, $\mu+2\overline{\rho}$ and $\nu +2\overline{%
\rho}$ in the same $W$-orbit but $H_{\mu}\neq H_{\nu}$.\ So the hypothesis $%
\mu+2\overline{\rho}\in P_{+}$ is crucial in the above proposition.
\end{Rem}

\section{The classical Lie algebras}

\subsection{Proof of the conjecture for $\mathfrak{gl}_{n}$}

\label{Sec_class}We now prove our conjecture in type $A$. We shall work in
fact with $\mathfrak{gl}_{n}$ rather than $\mathfrak{sl}_{n}$. The main tool
is a duality result between the branching coefficients $m_{\mu}^{\lambda}$
and some generalized Littlewood Richardson coefficients together with the
main result of \cite{Raj}. Each partition $\lambda=(\lambda_{1}\geq\cdots
\geq\lambda_{d}\geq0)$ can be regarded as a dominant weight of $\mathfrak{gl}%
_{n}$ by adding $n-d$ coordinates equal to $0$. We will use this convention
in this section.\ For any partition $\mu=(\mu_{1}\geq\cdots\geq\mu_{d}),$ we
have in fact%
\begin{equation}
s_{\mu}=\sum_{\lambda=(\lambda_{1}\geq\cdots\geq\lambda_{d}\geq0)}K_{%
\lambda,\mu}^{-1}h_{\lambda}   \label{JT}
\end{equation}
that is, the coefficients appearing in the expansion of $s_{\mu}$ on the $h$%
-basis are inverse Kostka numbers indexed by pairs $(\lambda,\mu)$ of
partitions with at most $d$ nonzero parts. When $\mathfrak{g}=\mathfrak{gl}%
_{n}$, the $h$-functions have also an additional property (which does not
hold for the other root systems).\ Consider $\beta=(\beta_{1},\ldots,%
\beta_{n})\in\mathbb{Z}_{\geq0}^{n}$, then $h_{\beta}=h_{\beta_{1}}\times%
\cdots\times h_{\beta_{n}}$.

Recall that the dominant weights of $\mathfrak{gl}_{n}$ can be regarded as
non increasing sequences of integers (possibly negative) with length $n$. We
will realise $\overline{\mathfrak{g}}=\mathfrak{gl}_{m_{1}}\oplus \cdots
\oplus \mathfrak{gl}_{m_{r}}$ as the subalgebra of $\mathfrak{gl}_{m}$ of
block matrices with block sizes $m_{1},\ldots ,m_{r}$.\ Now consider $\mu
\in \overline{P}$ such that $\mu =\mu ^{(1)}+\cdots +\mu ^{(r)}$ where $\mu
^{(k)}\in P_{+}^{(k)}$ as in \S\ \ref{subsec_Dec}. Then each $\mu ^{(k)}$ is
a non increasing sequence of integers of length $m_{k}$. We will assume
temporary that the coordinates of $\mu $ are nonnegative so that each $\mu
^{(k)}$ is a partition with $m_{k}$ parts.\ We then have according to (\ref%
{H-product})%
\begin{equation*}
H_{\mu }=\sum_{\lambda ^{(1)}\in P_{+}^{(1)}}\cdots \sum_{\lambda ^{(r)}\in
P_{+}^{(r)}}K_{\lambda ^{(1)},\mu ^{(1)}}^{-1}\cdots K_{\lambda ^{(r)},\mu
^{(r)}}^{-1}h_{\lambda ^{(1)}+\cdots +\lambda ^{(r)}}
\end{equation*}%
where each $\lambda ^{(k)}$ is a partition.\ In particular, we have $%
h_{\lambda ^{(1)}+\cdots +\lambda ^{(r)}}=h_{\lambda ^{(1)}}\times \cdots
\times h_{\lambda ^{(r)}}$ which yields%
\begin{equation*}
H_{\mu }=\prod_{i=1}^{k}\left( \sum_{\lambda ^{(k)}\in
P_{+}^{(k)}}K_{\lambda ^{(k)},\mu ^{(k)}}^{-1}h_{\lambda ^{(k)}}\right) .
\end{equation*}%
Finally by using (\ref{JT}), we obtain%
\begin{equation*}
H_{\mu }=\prod_{i=1}^{k}s_{\mu ^{(k)}}.
\end{equation*}%
We can now prove our conjecture for induced representations of $\mathfrak{gl}%
_{n}$

\begin{Prop}
Let $\mu$ and $\nu$ be dominant weights of $\overline{\mathfrak{g}}$.
Assume $H_{\mu}=H_{\nu}$. Then, there exists a permuation $\sigma$ of $%
\{1,\ldots,n\}$ such that $\sigma(\overline{R}_{+})=\overline{R}_{+}$.
\end{Prop}

\begin{proof}
By Theorem \ref{Th_multi}, we have $m_{\mu}^{\lambda}=\sum_{\sigma\in
S_{n}}\varepsilon(\sigma)\overline{\mathcal{P}}(\sigma(\lambda+\rho)-\mu-%
\rho )$.\ Set $\delta=(1,\ldots,1)\in\mathbb{Z}^{n}$. Since $\delta$ is
fixed by $S_{n}$, we have for any nonnegative integer $a$, $%
m_{\lambda+a\delta}^{\mu+a\delta}=m_{\lambda}^{\mu}$. Observe also that $%
P_{+}$ is invariant by translation by $\delta$.\ Therefore 
\begin{equation*}
H_{\mu+\delta a}=\sum_{\nu\in P_{+}}m_{\nu}^{\mu+a\delta}s_{\nu}=\sum
_{\lambda\in
P_{+}}m_{\lambda+a\delta}^{\mu+a\delta}s_{\lambda+a\delta}=\sum_{\lambda\in
P_{+}}m_{\lambda}^{\mu}s_{\lambda+a\delta}
\end{equation*}
by setting $\nu=\lambda+a\delta$ in the leftmost sum. So $H_{\mu}=H_{\nu}$
if and only if $H_{\mu+a\delta}=H_{\nu+a\delta}$. We can now choose $a$
sufficiently large so that $\mu\in\mathbb{Z}_{>0}^{n}$ and $\nu\in \mathbb{Z}%
_{>0}^{n}$. Decompose $\mu=\mu^{(1)}+\cdots+\mu^{(r)}$ and $\nu
=\nu^{(1)}+\cdots+\nu^{(r)}$ as in \S\ \ref{subsec_Dec}.\ For any $%
k=1,\ldots,r$, set $\delta_{k}=(1,\ldots,1)\in\mathbb{Z}^{m_{k}}$.\ The
similar decompositions of $\mu+a\delta$ and $\nu+a\delta$ verify $(\mu
+a\delta)^{(k)}=\mu^{(k)}+a\delta^{(k)}$ and $(\nu+a\delta)^{(k)}=\nu
^{(k)}+a\delta^{(k)}$ for any $k=1,\ldots,r$.\ We thus obtain%
\begin{equation*}
\prod_{i=1}^{k}s_{\mu^{(k)}+a\delta^{(k)}}=\prod_{i=1}^{k}s_{\nu^{(k)}+a%
\delta^{(k)}}. 
\end{equation*}
Now by the main result of \cite{Raj}, since the partitions $%
\mu^{(k)}+a\delta^{(k)}$ and $\nu^{(k)}+a\delta^{(k)}$ appearing above have
positive parts, we know that the set of partitions $\{\mu^{(k)}+a%
\delta^{(k)},k=1,\ldots r\}$ and $\{\nu^{(k)}+a\delta^{(k)},k=1,\ldots r\}$
coincide. There thus exists a permuation $\tau\in S_{r}$ such that $%
\mu ^{(k)}+a\delta^{(k)}=\nu^{(\tau(k))}+a\delta^{(\tau(k))}$.\ The
permuation $\tau$ preserves the lengths of the partitions so $%
m_{k}=m_{\tau(k)}$ and $\delta^{(k)}=\delta^{(\tau(k))}$ for any $%
k=1,\ldots,r$. We obtain $\mu ^{(k)}=\nu^{(\tau(k))}$.\ For any $%
k=1,\ldots,r,$ set $I_{k}=\{m_{k-1}+1,\ldots,m_{k}\}$ (with $m_{0}=0\}$.\
Then $I_{k}$ and $I_{\tau(k)}$ have the same cardinality because $%
m_{k}=m_{\tau(k)}$. Let $\sigma\in S_{n}$ be such that $%
\sigma(m_{k-1}+j)=m_{\tau(k)-1}+j$ for any $j\in\{1,\ldots,k\}$ and any $%
k\in\{1,\ldots,r\}$. Then $\sigma$ is a Dynkin diagram automorphism of $%
\overline{\mathfrak{g}}$. We have $\sigma(\mu)=\nu$ and $\sigma(\overline {R}%
_{+})=\overline{R}_{+}$ as desired.
\end{proof}

\bigskip

\begin{Rem}
\label{rem_Ce}Observe that we can have $\mu$ and $\nu$ in the same $W$%
-orbit, $\mu+2\overline{\rho}$ and $\nu+2\overline{\rho}$ in the same $W$%
-orbit but $H_{\mu}\neq H_{\nu}$. Consider for example $\overline{\mathfrak{g%
}}=\mathfrak{gl}_{4}\oplus\mathfrak{gl}_{2}$ in $\mathfrak{gl}_{6}$ and $%
\mu=(5,2,2,1\mid4,3)$ and $\nu=(5,4,3,1\mid2,2)$. We have $2\overline{\rho }%
=(3,1,-1,-3\mid1,-1)$ so $\mu+2\overline{\rho}=(8,3,1,-2\mid5,2)$ and $\nu+2%
\overline{\rho}=(8,5,2,-2\mid3,1)$ belong to the same $W$-orbit. By the
previous proposition, we have $H_{\mu}\neq H_{\nu}$. We cannot apply
Proposition \ref{Prop_mu+2ro} since neither $\mu+2\overline{\rho}$ or $\nu+2%
\overline{\rho}$ belongs to $P_{+}$.
\end{Rem}

\subsection{Polarisation}

Assume $\mathfrak{g=so}_{2n+1},\mathfrak{sp}_{2n}$ or $\mathfrak{so}_{2n}$
and $\overline{\mathfrak{g}}=\mathfrak{gl}_{n}$. Each dominant weight $\mu
\in\overline{P}_{+}$ defines a pair of partitions $(\mu_{+},\mu_{-})$ of
length $\leq n$ obtained by ordering decreasingly the positive and negative
coordinates of $\mu$, respectively. Recall also that to each partition $%
\lambda$ of length $\leq n$ corresponds a dominant weight of $P_{+}$. The
branching coefficients $m_{\mu}^{\lambda}$ were obtained by Littelwood (see 
\cite{Li}). They can be expressed in terms of the Littelwood-Richardson
coefficients as follows:%
\begin{equation*}
m_{\mu}^{\lambda}=%
\begin{cases}
\sum_{\gamma,\delta}c_{\mu_{+},\mu_{-}}^{\gamma}c_{\gamma,\delta}^{\lambda}
& \mbox{ for $\mathfrak{g=so}_{2n+1}$}, \\ 
\sum_{\gamma,\delta}c_{\mu_{+},\mu_{-}}^{\gamma}c_{\gamma,2\delta}^{\lambda}
& \mbox{ for $\mathfrak{g=sp}_{2n}$}, \\ 
\sum_{\gamma,\delta}c_{\mu_{+},\mu_{-}}^{\gamma}c_{\gamma,(2\delta)^{%
\ast}}^{\lambda} & \mbox{ for $\mathfrak{g=so}_{2n}$},%
\end{cases}
\end{equation*}
where $\gamma$ and $\delta$ runs over the set of partitions with length $%
\leq n$ and $(2\delta)^{\ast}$ is the conjugate partition of $2\delta$.

\begin{Prop}
Conjecture \ref{conj} is true for $\mathfrak{g=so}_{2n+1},\mathfrak{sp}_{2n}$
or $\mathfrak{so}_{2n}$ and $\overline{\mathfrak{g}}=\mathfrak{gl}_{n}$.
\end{Prop}

\begin{proof}
Consider $\mu$ and $\nu$ in $\overline{P}_{+}$ such that $H_{\mu}=H_{\nu} $.
We have $m_{\mu}^{\lambda}=m_{\nu}^{\lambda}$ for any $\lambda\in P_{+}$.
For any partition $\lambda$, write $\left\vert \lambda\right\vert $ the size
of $\lambda$, that is the sum of its parts. Observe first that $%
m_{\mu}^{\lambda }=0$ when $\left\vert \lambda\right\vert <\left\vert
\mu_{+}\right\vert +\left\vert \mu_{-}\right\vert $.\ Also, when $\left\vert
\lambda\right\vert =\left\vert \mu_{+}\right\vert +\left\vert
\mu_{-}\right\vert $ in the above branching coefficients, we get $%
\delta=\emptyset,\gamma=\lambda$ and $m_{\mu
}^{\lambda}=c_{\mu_{+},\mu_{-}}^{\lambda}$for $\mathfrak{g=so}_{2n+1},%
\mathfrak{sp}_{2n}$ or $\mathfrak{so}_{2n}$.

Assume $\left\vert \mu_{+}\right\vert +\left\vert \mu_{-}\right\vert
<\left\vert \nu_{+}\right\vert +\left\vert \nu_{-}\right\vert $.\ Then for $%
\lambda=\mu_{+}+\mu_{-}$, we have $m_{\mu}^{\lambda}=c_{\mu_{+},\mu_{-}}^{%
\lambda}=1$ whereas $m_{\nu}^{\lambda}=0$ since $\left\vert \lambda
\right\vert =\left\vert \mu_{+}\right\vert +\left\vert \mu_{-}\right\vert
<\left\vert \nu_{+}\right\vert +\left\vert \nu_{-}\right\vert $. So we
obtain a contradiction. Similarly, we cannot have $\left\vert
\mu_{+}\right\vert +\left\vert \mu_{-}\right\vert >\left\vert
\nu_{+}\right\vert +\left\vert \nu_{-}\right\vert $. Therefore $\left\vert
\mu_{+}\right\vert +\left\vert \mu_{-}\right\vert =\left\vert
\nu_{+}\right\vert +\left\vert \nu _{-}\right\vert $. Then for any $\lambda$
such that $\left\vert \lambda \right\vert =\left\vert \mu_{+}\right\vert
+\left\vert \mu_{-}\right\vert =\left\vert \nu_{+}\right\vert +\left\vert
\nu_{-}\right\vert $, we have $c_{\mu_{+},\mu_{-}}^{\lambda}=c_{\nu_{+},%
\nu_{-}}^{\lambda}$.\ By the main result of \cite{Raj}, we obtain the
equality of sets $\{\mu_{+},\mu _{-}\}=\{\nu_{+},\nu_{-}\}$. When $%
\mu_{+}=\nu_{+}$ and $\mu_{-}=\nu_{-}$, we have $\mu=\nu$ and the conjecture
holds. When $\mu_{+}=\nu_{-}$ and $\mu _{-}=\nu_{+},$ we have $\mu=-%
\overline{w}_{0}\nu$ where $\overline{w}_{0}$ is the longest element of $%
\overline{W}$ that is, the permutation of $\{1,\ldots,n\}$ such that $%
w_{0}(k)=n-k+1$. Since $-\overline{w}_{0}\in W$ and $-w_{0}(\overline{R}%
_{+})=\overline{R}_{+}$ we are done.
\end{proof}


\label{SecFinal}We now summarize our results.

\begin{Th}
\label{Th_final}Consider $\mu,\nu\in\overline{P}_{+}$.

\begin{enumerate}
\item When $\mu$ and $\nu$ are conjugate under the action of a Dynkin
diagram automorphism of $\overline{\mathfrak{g}}$ lying in $W$, we have $%
H_{\mu }=H_{\nu}$.

\item Conversely, if we assume $H_{\mu}=H_{\nu},$ then $\mu$ and $\nu$ are
conjugate under the action of a Dynkin diagram automorphism lying in $W$
when one of the following hypotheses is satisfied

\begin{itemize}
\item $\mu$ and $\nu$ belong to the same Weyl chamber of $\mathfrak{g}$ (in
which case $\mu=\nu$),

\item $\mu$ and $\nu$ are far enough of the walls of the Weyl chamber where
they appear (each set $E_{\mu}$ or $E_{\nu}$ is entirely contained in a Weyl
chamber),

\item $\mu+2\overline{\rho}$ or $\nu+2\overline{\rho}$ belongs to $P_{+}$,

\item $\mathfrak{g}=\mathfrak{gl}_{n}$,

\item $\mathfrak{g=so}_{2n+1},\mathfrak{sp}_{2n}$ or $\mathfrak{so}_{2n}$
and $\overline{\mathfrak{g}}=\mathfrak{gl}_{n}$.
\end{itemize}
\end{enumerate}
\end{Th}

\end{document}